\documentclass[11pt]{amsart}

\usepackage{amsmath,amssymb,amsthm,mathtools}
\usepackage[margin=1in]{geometry}
\usepackage[colorlinks=true,linkcolor=blue,citecolor=blue,urlcolor=blue]{hyperref}

\theoremstyle{plain}
\newtheorem{theorem}{Theorem}[section]
\newtheorem{proposition}[theorem]{Proposition}
\newtheorem{lemma}[theorem]{Lemma}
\newtheorem{corollary}[theorem]{Corollary}
\newtheorem{question}[theorem]{Question}
\theoremstyle{definition}

\newtheorem{remark}[theorem]{Remark}

\newcommand{\Z}{\mathbb{Z}}
\newcommand{\R}{\mathbb{R}}
\newcommand{\C}{\mathbb{C}}
\newcommand{\E}{\mathbb{E}}
\newcommand{\PP}{\mathbb{P}}
\renewcommand{\Re}{\operatorname{Re}}
\renewcommand{\Im}{\operatorname{Im}}
\DeclareMathOperator{\dist}{dist}
\DeclareMathOperator{\spann}{span}

\DeclareMathOperator{\sgn}{sgn}

\newcommand{\dimR}{\dim_{\R}}
\newcommand{\dimC}{\dim_{\C}}
\newcommand{\mult}{\operatorname{mult}}
\newcommand{\Lone}{L^{1}}
\newcommand{\Lonez}{L^{1}_{0}}
\newcommand{\hRok}{h^{\mathrm{Rok}}}
\newcommand{\actson}{\curvearrowright}

\begin{document}
\title[Positive Rokhlin entropy and $L^1$-orbit multiplicity]
{Positive Rokhlin Entropy Implies Infinite $L^1$-Orbit Multiplicity:
A Negative Answer to Thouvenot's Question}

\author[Z.~Hu]{Zongrui Hu}
\address{School of Mathematical Sciences, University of Science and Technology of China, Hefei, Anhui 230026, P.R. China}
\email{zongrui@mail.ustc.edu.cn}

\author[L.~Xu]{Leiye Xu}
\address{School of Mathematical Sciences, University of Science and Technology of China, Hefei, Anhui 230026, P.R. China}
\email{leoasa@mail.ustc.edu.cn}

\author[S.~Zhang]{Shuhao Zhang}
\address{School of Mathematical Sciences, University of Science and Technology of China, Hefei, Anhui 230026, P.R. China}
\email{yichen12@mail.ustc.edu.cn}
\date{\today}



\begin{abstract}
We prove that every free ergodic measure-preserving action of a countably infinite amenable group with positive Rokhlin entropy has infinite complex $L^1$-orbit multiplicity, both on $L^1$ and on its mean-zero subspace $L^1_0$. This gives a negative answer to a question of J.-P. Thouvenot recorded by Iwanik and establishes the corresponding endpoint statement at $p=1$ of Iwanik's theorem that positive entropy implies infinite $L^p$-multiplicity for every $p>1$. The proof combines Malykhin's rigidity theorem for independent random variables, Seward's Bernoulli factor theorem, and a Følner set argument.

\end{abstract}

\maketitle

\section{Introduction}\label{sec:introduction}

Let $\Gamma$ be a countable group acting by measure-preserving transformations
on a standard probability space $(X,\mathcal B,\mu)$. The corresponding
Koopman isometries on complex $L^1(X,\mu)$ are given by
$U_\gamma f=f\circ\gamma^{-1}$. For $f_1,\dots,f_m\in L^1(X;\C)$, put
\begin{equation}\label{eq:orbit-space}
 G_\Gamma(f_1,\dots,f_m)
 :=\overline{\operatorname{span}_{\C}\{U_\gamma f_s:
 \gamma\in\Gamma,\ 1\le s\le m\}}^{\,L^1}.
\end{equation}
A function $f\in L^1(X;\C)$ is an \emph{$L^1$-cyclic vector} if
$G_\Gamma(f)=L^1(X)$. Define
\begin{equation}\label{eq:L1-multiplicity}
\begin{aligned}
 \mult_{\Lone}(\Gamma\actson X):=\inf\bigl\{m\ge1:\;&
 G_\Gamma(f_1,\dots,f_m)=L^1(X)\\
 &\text{for some }f_1,\dots,f_m\in L^1(X;\C)\bigr\},
\end{aligned}
\end{equation}
where the value is $\infty$ if no finite family exists. We use the same
definition on
\[
 L^1_0(X):=\left\{f\in L^1(X):\int_X f\,d\mu=0\right\}
\]
and denote the resulting multiplicity by
$\mult_{\Lonez}(\Gamma\actson X)$.

Iwanik~\cite[p.~91]{Iw2} recorded Thouvenot's question whether every ergodic
automorphism has an $L^1$-cyclic vector. The question was still open in the
survey of Kanigowski and Lema\'nczyk~\cite[p.~6]{KL}. Iwanik~\cite{Iw1}
proved that positive entropy implies infinite $L^p$-multiplicity for every
$p>1$, by an argument that does not extend to $p=1$. We prove the corresponding
result at $p=1$, and more generally for free ergodic actions of countably
infinite amenable groups.

\begin{theorem}\label{thm:main}
Let $\Gamma$ be a countably infinite amenable group and let
$\Gamma\actson(X,\mathcal B,\mu)$ be a free ergodic p.m.p. action on a standard
probability space. If
\[
 \hRok(\Gamma\actson X)>0,
\]
then
\[
 \mult_{\Lone}(\Gamma\actson X)
 =\mult_{\Lonez}(\Gamma\actson X)=\infty.
\]
\end{theorem}

For free ergodic actions of countably infinite amenable groups, Rokhlin entropy
agrees with Kolmogorov--Sinai entropy~\cite{OW,Sew1}. Hence, for
$\Gamma=\Z$, Theorem~\ref{thm:main} answers Thouvenot's question in the
negative for every free ergodic transformation of positive entropy. It also
shows that no finite family of functions has dense Koopman orbit span. In
particular, every nontrivial Bernoulli shift over a countably infinite amenable
group has infinite $L^1$- and $L^1_0$-orbit multiplicity; see
Corollary~\ref{cor:bernoulli}.

Burguet and Shi~\cite{BS} introduced the related notion of
\emph{topological multiplicity}, defined via the composition operator on
$C(Y)$ for a homeomorphism \(T\) of a compact metrizable space \(Y\). If $\mu$ is a
$T$-invariant Borel probability measure, then every topological generating
family, after complexification, also generates $L^1(Y,\mu;\C)$. Hence
\[
 \mult_{\Lone}(\Z\actson(Y,\mu))
 \le \operatorname{Mult}_{\mathrm{top}}(T).
\]
Thus finite topological multiplicity is an a priori stronger condition than
finite $L^1$-orbit multiplicity, and the entropy theorem of Burguet and Shi
does not by itself settle the $L^1$ endpoint considered here. Our proof uses
instead a Bernoulli factor together with an $L^1$ rigidity estimate for
independent random variables. Moreover, when \(\Gamma=\mathbb Z\),
Theorem~\ref{thm:main}, combined with the variational principle, gives
another proof of
\[
 h_{\mathrm{top}}(T)>0
 \quad\Longrightarrow\quad
 \operatorname{Mult}_{\mathrm{top}}(T)=\infty.
\]

To prove the theorem, suppose that $m$ Koopman orbits generate $L^1$.
Seward's theorem gives a Bernoulli factor with $r=2m+1$ independent binary
coordinates and base entropy smaller than $\hRok(\Gamma\actson X)$. Approximate
these coordinates by finite orbit sums and translate the approximants over a
F\o lner set $F$. The resulting $r|F|$ independent normalized random variables
are arbitrarily close to a real subspace of dimension at most $2m|FK|$.
Choosing $F$ so that $|FK|<(1+\delta)|F|$ contradicts
Theorem~\ref{thm:malykhin}.

\section{Preliminaries}\label{sec:prelim}

\subsection{Independent random variables}\label{ss:rigidity}

All linear spaces in this subsection are real. For a finite family
$\{\xi_1,\dots,\xi_N\}\subset\Lone(\Omega,\PP;\R)$ and a real linear subspace
$W\subseteq\Lone(\Omega,\PP;\R)$, write
\[
 \dist_{\Lone}(\xi,W):=\inf_{w\in W}\|\xi-w\|_1.
\]
We use the following case of Malykhin's rigidity theorem
\cite[Theorem~1.1]{Mal}.

\begin{theorem}[Malykhin]\label{thm:malykhin}
For every $\varepsilon\in(0,1)$ there is $c_\varepsilon>0$ such that, whenever
$\xi_1,\dots,\xi_N$ are independent real random variables satisfying
$\E\xi_i=0$ and $\E|\xi_i|=1$, and
$W\subseteq\Lone(\Omega,\PP;\R)$ is a linear subspace with
$\dimR W\le(1-\varepsilon)N$, one has
\begin{equation}\label{eq:malykhin}
 \frac1N\sum_{i=1}^N\dist_{\Lone}(\xi_i,W)\ge c_\varepsilon.
\end{equation}
\end{theorem}

We shall apply this theorem to real parts of complex orbit spaces.

\begin{lemma}\label{lem:c2r}
Let $V\subseteq\Lone(\Omega;\C)$ be a complex linear space of dimension $d$ and
let $W=\{\Re v:v\in V\}$. Then $W$ is a real linear subspace,
$\dimR W\le2d$, and
\begin{equation}\label{eq:pointwise}
 \|\xi-\Re v\|_1\le\|\xi-v\|_1
\end{equation}
for every real-valued $\xi\in\Lone(\Omega;\R)$ and every $v\in V$.
\end{lemma}

\begin{proof}
If $v_1,\dots,v_d$ is a complex basis of $V$, then
\[
 W\subseteq\operatorname{span}_{\R}
 \{\Re v_j,\Im v_j:1\le j\le d\},
\]
so $\dimR W\le2d$. For real $\xi$ one has
$\xi-\Re v=\Re(\xi-v)$, and \eqref{eq:pointwise} follows pointwise.
\end{proof}

\subsection{Amenability and F\o lner sets}\label{ss:folner}
Recall that a countable group $\Gamma$ is \emph{amenable} if and only if it
admits a (right) F\o lner sequence; consequently, for every finite set
$K\subseteq\Gamma$ and every $\delta>0$ there is a nonempty finite set
$F\subseteq\Gamma$ with $\lvert FK\setminus F\rvert\le\delta\lvert F\rvert$. We
use this in the following packaged form; see, e.g., Kerr and Li
\cite[Ch.~4]{KeLi}.

\begin{lemma}[F\o lner packing]\label{lem:folner}
Let $\Gamma$ be a countable amenable group, let $K\subseteq\Gamma$ be finite with
$e\in K$, and let $\delta>0$. Then there is a nonempty finite set
$F\subseteq\Gamma$ with $\lvert FK\rvert\le(1+\delta)\lvert F\rvert$.
\end{lemma}

\begin{proof}
By amenability choose finite nonempty $F$ with
$\lvert FK\setminus F\rvert\le\delta\lvert F\rvert$ \cite[Ch.~4]{KeLi}. Since
$e\in K$ we have $F\subseteq FK$, whence
$\lvert FK\rvert=\lvert F\rvert+\lvert FK\setminus F\rvert\le(1+\delta)\lvert F\rvert$.
\end{proof}

\subsection{Rokhlin entropy and Bernoulli factors}\label{ss:entropy}

Let \(\Gamma\) be a countable group acting by measure-preserving transformations
on a standard probability space \((X,\mathcal B,\mu)\).
For a countable Borel partition \(\alpha=\{A_i\}_{i\in I}\), its \emph{Shannon entropy} is
\[
H(\alpha) := -\sum_{A\in\alpha}\mu(A)\log\mu(A),
\]
with the convention \(0\log 0=0\).  The \emph{Rokhlin entropy} of the action
\(\Gamma\actson(X,\mu)\) is then defined as the infimum
\[
h_{\Gamma}^{\mathrm{Rok}}(X, \mu) = \inf \bigl\{ H(\alpha \mid \mathcal{I}) :
\alpha \text{ is a countable partition and }
\sigma\text{-alg}_{\Gamma}(\alpha) \vee \mathcal{I} = \mathcal{B}(X) \bigr\},
\]
where \(\mathcal{I}\) is the \(\sigma\)-algebra of \(\Gamma\)-invariant Borel sets.
In the present paper we only consider ergodic actions, in which case the Rokhlin entropy simplifies to
\[
\hRok(\Gamma\actson X) := \inf\bigl\{ H(\alpha) : \alpha \text{ is a countable
Borel partition that generates } \mathcal B \text{ under } \Gamma \bigr\}.
\]
This notion was introduced by Seward~\cite{Sew1}. For free ergodic actions of
countably infinite amenable groups it coincides with the classical
Kolmogorov--Sinai entropy, as proved in \cite{Sew1,STD16}.  We shall only need the following factor theorem of Seward \cite{Sew2},
which generalizes Sinai's classical result to all countably infinite groups.

To present Seward's result, we recall that the \emph{Shannon entropy} of a probability space \((L,\lambda)\) is
\[
H(L,\lambda) = \sum_{\ell\in L} -\lambda(\ell)\log\lambda(\ell)
\]
if \(\lambda\) has countable support, and \(H(L,\lambda)=\infty\) otherwise. 

\begin{theorem}[Seward]\label{thm:sinai}
Let $\Gamma$ be a countably infinite group and let
$\Gamma\actson(X,\mathcal B,\mu)$ be a free ergodic p.m.p. action with
$\hRok(\Gamma\actson X)=h>0$. If $(L,\lambda)$ is a standard probability space
with finite Shannon entropy $H(L,\lambda)\le h$, then
$\Gamma\actson(L^\Gamma,\lambda^\Gamma)$ is a factor of
$\Gamma\actson(X,\mu)$.
\end{theorem}

\section{A F\o lner estimate for Bernoulli shifts}\label{sec:folner}

Fix $q\in(0,1)$ and let $\nu_q$ be the Bernoulli law on $\{0,1\}$ with
$\nu_q(\{1\})=q$. For an integer $r\ge1$ consider the Bernoulli shift over
$\Gamma$ with $r$-fold base,
\begin{equation}\label{eq:bernoulli}
(\Omega,\mathcal F,\PP)=\Bigl(\bigl(\{0,1\}^r\bigr)^{\Gamma},\ \bigl(\nu_q^{\otimes r}\bigr)^{\Gamma}\Bigr),
\qquad
(\gamma\cdot\omega)_g=\omega_{\gamma^{-1}g}\quad(\gamma,g\in\Gamma),
\end{equation}
with Koopman isometry $U_\gamma F=F\circ\gamma^{-1}$. For $1\le j\le r$ and
$g\in\Gamma$ let $B_{j,g}(\omega)=(\omega_g)_j$ be the $j$-th binary coordinate of
the symbol at $g$. The family $\{B_{j,g}:1\le j\le r,\ g\in\Gamma\}$ consists of
the distinct coordinate projections of a product measure and is therefore jointly
independent. A direct computation gives, for every $\gamma\in\Gamma$,
\begin{equation}\label{eq:shift}
(U_\gamma B_{j,e})(\omega)=B_{j,e}(\gamma^{-1}\cdot\omega)
=\bigl((\gamma^{-1}\cdot\omega)_e\bigr)_j=(\omega_\gamma)_j=B_{j,\gamma}(\omega),
\end{equation}
so $U_\gamma B_{j,e}=B_{j,\gamma}$. Define the centered, normalized variables
\begin{equation}\label{eq:xi}
\xi_{j,g}:=\frac{B_{j,g}-q}{2q(1-q)}\qquad(1\le j\le r,\ g\in\Gamma).
\end{equation}
Each $\xi_{j,g}$ is real. Since $\lvert B_{j,g}-q\rvert$ equals $1-q$ with
probability $q$ and equals $q$ with probability $1-q$, one has
$\E\lvert B_{j,g}-q\rvert=2q(1-q)$, and hence $\E\xi_{j,g}=0$ and
$\|\xi_{j,g}\|_1=1$. Moreover \eqref{eq:shift} and \eqref{eq:xi} give
\begin{equation}\label{eq:xishift}
U_\gamma\,\xi_{j,e}=\xi_{j,\gamma}\qquad(\gamma\in\Gamma).
\end{equation}

\begin{proposition}\label{prop:folner}
Let $\Gamma$ be a countably infinite amenable group and let
$f_1,\dots,f_m\in\Lone(\Omega;\C)$, with $G=G_\Gamma(f_1,\dots,f_m)$ the closed orbit
span \eqref{eq:orbit-space}. Suppose $r>2m$ and fix $\varepsilon$ with
$0<\varepsilon<1-\tfrac{2m}{r}$. Then, with $c_\varepsilon$ the constant of
Theorem~\ref{thm:malykhin},
\begin{equation}\label{eq:sep}
\frac1r\sum_{j=1}^{r}\dist_{\Lone}(\xi_{j,e},G)\ \ge\ c_\varepsilon.
\end{equation}
The same holds when the $f_s$ lie in $\Lonez(\Omega)$ and $G$ is their closed
orbit span in $\Lonez(\Omega)$.
\end{proposition}

\begin{proof}
Fix $\tau>0$. For each $j\in\{1,\dots,r\}$ the algebraic orbit span of the $f_s$
is dense in $G$, so there is a finite orbit sum
\begin{equation}\label{eq:gj}
g_j=\sum_{s=1}^m\sum_{k\in K_0}c_{j,s,k}\,U_k f_s
\qquad\text{with}\qquad
\|\xi_{j,e}-g_j\|_1<\dist_{\Lone}(\xi_{j,e},G)+\tau.
\end{equation}
After enlarging the finitely many index sets and setting unused coefficients
to zero, the same finite set $K_0\subseteq\Gamma$, with $e\in K_0$, may be
used for all $j$.

Since $2m<(1-\varepsilon)r$, fix $\delta>0$ with $2m(1+\delta)\le(1-\varepsilon)r$.
By Lemma~\ref{lem:folner} choose a nonempty finite set $F\subseteq\Gamma$ with
$\lvert FK_0\rvert\le(1+\delta)\lvert F\rvert$. Set
\begin{equation}\label{eq:VF}
V_F:=\spann_{\C}\{U_g f_s:\ g\in FK_0,\ 1\le s\le m\},\qquad
\dimC V_F\le m\lvert FK_0\rvert\le m(1+\delta)\lvert F\rvert,
\end{equation}
and $W_F:=\{\Re v:v\in V_F\}$. By Lemma~\ref{lem:c2r},
\begin{equation}\label{eq:dimWF}
\dimR W_F\le 2m(1+\delta)\lvert F\rvert\le(1-\varepsilon)\,r\,\lvert F\rvert.
\end{equation}

The $r\lvert F\rvert$ variables
$\{\xi_{j,h}:1\le j\le r,\ h\in F\}$ are independent, real, centered,
and $\Lone$-normalized. Theorem~\ref{thm:malykhin} and
\eqref{eq:dimWF} give
\begin{equation}\label{eq:applymalykhin}
c_\varepsilon\ \le\ \frac1{r\lvert F\rvert}\sum_{j=1}^{r}\sum_{h\in F}\dist_{\Lone}(\xi_{j,h},W_F).
\end{equation}
Fix $j$ and $h\in F$. From \eqref{eq:gj},
$U_h g_j=\sum_{s}\sum_{k\in K_0}c_{j,s,k}\,U_{hk}f_s$, and each index $hk$ with
$k\in K_0$ lies in $FK_0$; hence $U_h g_j\in V_F$, so $\Re(U_h g_j)\in W_F$. Since
$\xi_{j,h}$ is real, Lemma~\ref{lem:c2r}, identity \eqref{eq:xishift}, and the
$\Lone$-isometry of $U_h$ give
\begin{align*}
\dist_{\Lone}(\xi_{j,h},W_F)
&\le \|\xi_{j,h}-\Re(U_hg_j)\|_1\\
&\le \|\xi_{j,h}-U_hg_j\|_1\\
&=\|U_h(\xi_{j,e}-g_j)\|_1\\
&=\|\xi_{j,e}-g_j\|_1\\
&<\dist_{\Lone}(\xi_{j,e},G)+\tau.
\end{align*}
Substituting into \eqref{eq:applymalykhin} bounds its right-hand side by
$\frac1r\sum_{j=1}^r\dist_{\Lone}(\xi_{j,e},G)+\tau$. Letting $\tau\downarrow0$
proves \eqref{eq:sep}. If the generators lie in $\Lonez$, then their finite orbit
sums and closed orbit span lie in $\Lonez$, and each $\xi_{j,e}$ has mean
zero. The same proof applies.
\end{proof}

\begin{remark}\label{rem:direct}
With $r=3$, $m=1$, and $q=\tfrac12$, Proposition~\ref{prop:folner} shows
that the uniform eight-symbol Bernoulli shift $(\{0,1\}^3)^\Gamma$ over a
countably infinite amenable group has no complex $L^1$-cyclic vector.
\end{remark}

\section{Proof of the main theorem}\label{sec:factors}

\begin{lemma}[Multiplicity passes to factors]\label{lem:factor}
Let $\pi\colon(X,\mathcal B,\mu)\to(Y,\mathcal C,\nu)$ be a factor map between
p.m.p.\ $\Gamma$-systems. Then
\[
\mult_{\Lone}(\Gamma\actson Y)\le\mult_{\Lone}(\Gamma\actson X)
\qquad\text{and}\qquad
\mult_{\Lonez}(\Gamma\actson Y)\le\mult_{\Lonez}(\Gamma\actson X).
\]
\end{lemma}

\begin{proof}
Identify $\Lone(Y,\nu)$ with the closed subspace of $\Lone(X,\mu)$ by
$h\mapsto h\circ\pi$. Let
$P=\E(\,\cdot\mid\pi^{-1}\mathcal C)$. Then $P$ is an $L^1$-contraction,
preserves integrals, fixes $L^1(Y)$, and commutes with every $U_\gamma$
\cite[Ch.~5]{EW}.

If $f_1,\dots,f_m$ generate $\Lone(X)$, applying $P$ to their orbit span gives
\[
\Lone(Y)=P(\Lone(X))\subset\overline{\operatorname{span}}\{U_\gamma(Pf_i)\},
\]
so $Pf_1,\dots,Pf_m$ generate $\Lone(Y)$. Hence
$\mult_{\Lone}(\Gamma\actson Y)\le m$. Since $P$ preserves integrals, the
same proof applies to $\Lonez$.
\end{proof}
\begin{proof}[Proof of Theorem~\ref{thm:main}]
Write $h=\hRok(\Gamma\actson X)>0$, and let $H(q)=-q\log q-(1-q)\log(1-q)$ be the
binary entropy in the logarithm base fixing $h$. Fix $m\ge1$ and put $r=2m+1$.
Since $H(q)\to0$ as $q\downarrow0$, choose $q\in(0,\tfrac12)$ with $rH(q)<h$. The
$r$-fold Bernoulli base $(\{0,1\}^r,\nu_q^{\otimes r})$ has Shannon entropy
$H(\nu_q^{\otimes r})=rH(q)$ by additivity over independent coordinates. As
$\Gamma\actson X$ is free ergodic with $h>0$, Theorem~\ref{thm:sinai} provides a
factor map $\pi\colon X\to\Omega$ onto the Bernoulli shift $\Omega$ of
\eqref{eq:bernoulli}.

Suppose $\mult_{\Lone}(\Gamma\actson X)\le m$. By Lemma~\ref{lem:factor},
$\mult_{\Lone}(\Gamma\actson\Omega)\le m$, so there are $f_1,\dots,f_m\in\Lone(\Omega)$
with $G_\Gamma(f_1,\dots,f_m)=\Lone(\Omega)$. In particular each $\xi_{j,e}$ lies in this
span, so $\dist_{\Lone}(\xi_{j,e},G)=0$ for $1\le j\le r$. But
Proposition~\ref{prop:folner}, applicable because $r=2m+1>2m$, gives for any
$\varepsilon\in(0,1-\tfrac{2m}{r})$
\[
0=\frac1r\sum_{j=1}^r\dist_{\Lone}(\xi_{j,e},G)\ \ge\ c_\varepsilon>0,
\]
a contradiction. Hence $\mult_{\Lone}(\Gamma\actson X)>m$; as $m\ge1$ was
arbitrary, $\mult_{\Lone}(\Gamma\actson X)=\infty$.

The same argument applies to $\Lonez$: each $\xi_{j,e}$ has mean zero,
Lemma~\ref{lem:factor} preserves the mean-zero spaces, and the last assertion
of Proposition~\ref{prop:folner} gives the same contradiction.
\end{proof}

\begin{corollary}\label{cor:bernoulli}
Every nontrivial Bernoulli shift over a countably infinite amenable group has
infinite complex $\Lone$- and $\Lonez$-orbit multiplicity.
\end{corollary}

\begin{proof}
Let $\Gamma\actson(L^\Gamma,\lambda^\Gamma)$ have nontrivial base. The
action is ergodic. To prove essential freeness, fix $\gamma\ne e$ and put
$H=\langle\gamma\rangle$. If $H$ is
infinite, then on every infinite $H$-orbit in $\Gamma$ the fixed-point relation
$\gamma\cdot\omega=\omega$ forces infinitely many independent coordinates to
be equal, an event of probability zero for a nontrivial base. If $H$ has finite
order $d\ge2$, then $\Gamma$ is the disjoint union of infinitely many right
$H$-cosets. On each coset the same fixed-point relation forces $d$ independent
coordinates to be equal. The probability of this event on one coset is
\[
 \rho_d:=\sum_{a\text{ an atom of }\lambda}\lambda(\{a\})^d<1,
\]
and the coset events are independent; hence the probability that all of them
occur is $\lim_{n\to\infty}\rho_d^n=0$. Thus the shift is essentially free.

For amenable $\Gamma$, its Rokhlin entropy agrees with its classical entropy,
which is the base Shannon entropy $H(L,\lambda)>0$; see~\cite{OW,Sew1}.
Theorem~\ref{thm:main} applies.
\end{proof}

\section{Amenability and zero entropy}\label{sec:boundary}

\subsection{Non-amenable groups}\label{ss:sofic}
Seward's theorem holds for every countably infinite group. The proof of
Theorem~\ref{thm:main}, however, also uses the estimate
\[
 |FK|\le(1+\delta)|F|.
\]
In a non-amenable group one may cannot obtain this estimate uniformly for an arbitrary finite set \(K\) arising from the orbit approximants. A sofic approximation
does not produce a finite subset of the group with this property, so the proof
gives no information about Bernoulli shifts over non-amenable groups. Nor does
the infinite multiplicity of the Bernoulli Koopman representation on $L^2_0$
settle the $L^1$ question, since density in $L^2$ implies density in $L^1$, not
conversely.

\begin{question}\label{q:nonamenable}
Does a nontrivial Bernoulli shift over the free group $F_2$ have an
$L^1$-cyclic vector? Is its $L^1$-orbit multiplicity finite or infinite?
\end{question}

\subsection{Zero entropy}\label{ss:zero}
Zero entropy is compatible with an $L^1$-cyclic vector: an irrational rotation
has a cyclic vector in $L^2$, whose orbit span is also dense in $L^1$. On the
other hand, zero-entropy systems may have purely absolutely continuous spectrum on
$L^2$. Thus $L^2$ spectral multiplicity alone does not determine $L^1$-orbit
multiplicity.

\begin{question}\label{q:zero}
Does there exist a free ergodic zero-entropy action of a countably infinite
amenable group with infinite $L^1$-orbit multiplicity?
\end{question}

A negative answer, together with Theorem~\ref{thm:main}, would characterize
finite $L^1$-orbit multiplicity among free ergodic amenable actions by zero
entropy. A natural substitute for independence is disjointness of supports.

\begin{lemma}\label{lem:disjoint}
Let $\eta_1,\dots,\eta_N\in L^1(X;\R)$ have pairwise disjoint supports and
$\|\eta_i\|_1=1$. If $W\subseteq L^1(X;\R)$ is a real linear subspace with
$\dim_{\R}W\le(1-\varepsilon)N$, then
\[
 \frac1N\sum_{i=1}^N\dist_{L^1}(\eta_i,W)\ge\varepsilon.
\]
\end{lemma}

\begin{proof}
Put $A_i=\{\eta_i\ne0\}$ and define $R:L^1(X;\R)\to\ell^1_N$ by
\[
 (Rf)_i=\int_{A_i}f\,\sgn(\eta_i)\,d\mu.
\]
Since the sets $A_i$ are disjoint, $\|Rf\|_{\ell^1}\le\|f\|_1$, and
$R\eta_i=e_i$. Hence
\[
 \dist_{L^1}(\eta_i,W)\ge\dist_{\ell^1}(e_i,RW).
\]
Let $V=RW$ and let $P_{V^\perp}$ be the Euclidean orthogonal projection onto
$V^\perp$. By $\ell^1$--$\ell^\infty$ duality,
\[
 \dist_{\ell^1}(e_i,V)\ge\|P_{V^\perp}e_i\|_2.
\]
Since $\|P_{V^\perp}e_i\|_2\le1$,
\[
\begin{aligned}
 \sum_{i=1}^N\dist_{\ell^1}(e_i,V)
 &\ge\sum_{i=1}^N\|P_{V^\perp}e_i\|_2^2\\
 &=\operatorname{tr}(P_{V^\perp})
 =N-\dim V
 \ge\varepsilon N.
\end{aligned}
\]
\end{proof}

The disjoint-support estimate cannot be applied to arbitrarily long pieces
of a fixed orbit.

\begin{proposition}\label{prop:recurrence}
Let \(T\colon(X,\mu)\to(X,\mu)\) be an invertible measure-preserving transformation, and let $0\ne\eta\in L^1(X)$. If the
supports of
\[
 \eta,U\eta,\dots,U^{M-1}\eta
\]
are pairwise disjoint, then
\[
 M\le\mu(\{\eta\ne0\})^{-1}.
\]
\end{proposition}

\begin{proof}
Set $A=\{\eta\ne0\}$. The supports of $\eta,U\eta,\dots,U^{M-1}\eta$ are
$A,TA,\dots,T^{M-1}A$. Hence
\[
 M\mu(A)=\mu\left(\bigcup_{k=0}^{M-1}T^kA\right)\le1.
\]
\end{proof}

Malykhin's theorem assumes independence, whereas Lemma~\ref{lem:disjoint}
assumes disjoint supports. The results of Malykhin and Ryutin~\cite{MR} suggest
studying unconditional families between these two cases. For orbit
multiplicity, however, one would need arbitrarily many such orbit families
jointly; one unconditional orbit is not sufficient.

\section{Further questions}\label{sec:further}

\begin{question}[Relative multiplicity]\label{q:relative}
Let $\mathcal A\subseteq\mathcal B(X)$ be a $\Gamma$-invariant
sub-$\sigma$-algebra and set
\[
 L^1_0(X\mid\mathcal A)
 :=\{f\in L^1(X):\E(f\mid\mathcal A)=0\}.
\]
For $f_1,\dots,f_m\in L^1_0(X\mid\mathcal A)$, let
\[
 \mathcal M_{\mathcal A}(f_1,\dots,f_m)
 :=\overline{\operatorname{span}\{aU_\gamma f_s:
 a\in L^\infty(\mathcal A),\ \gamma\in\Gamma,\ 1\le s\le m\}}^{\,L^1}.
\]
If the action has positive Rokhlin entropy relative to $\mathcal A$, can
$L^1_0(X\mid\mathcal A)$ equal
$\mathcal M_{\mathcal A}(f_1,\dots,f_m)$ for some finite family?
\end{question}

The main difficulty is a conditional form of Malykhin's theorem for the
disintegration over $\mathcal A$.

\begin{question}[Symmetric function spaces]\label{q:symmetric}
Let $E$ be a rearrangement-invariant Banach function space on a probability
space, and define $E$-orbit multiplicity by replacing $L^1$ with $E$ in
\eqref{eq:orbit-space} and \eqref{eq:L1-multiplicity}. For which spaces $E$
does positive entropy imply infinite $E$-orbit multiplicity?
\end{question}

The same argument extends to a rearrangement-invariant space \(E\) provided that conditional expectations are bounded on \(E\) and normalized independent families satisfy a uniform rigidity estimate of the required form.
 Related width estimates were obtained by Astashkin and
Malykhin~\cite{AM}.

The final question concerns a zero-entropy system with countable Lebesgue
spectrum. Let $T$ be the time-one map of the horocycle flow on the unit tangent
bundle of a compact hyperbolic surface; see~\cite{TA}. Recall that a family
$(x_i)_{i\in I}$ in a Banach space is unconditional if there is $C\ge1$ such
that
\[
 C^{-1}\left\|\sum_{i\in J}a_ix_i\right\|
 \le\left\|\sum_{i\in J}\varepsilon_i a_ix_i\right\|
 \le C\left\|\sum_{i\in J}a_ix_i\right\|
\]
for every finite $J\subseteq I$, all scalars $a_i$, and all
$\varepsilon_i\in\{-1,1\}$.

\begin{question}[Horocycle maps]\label{q:horocycle}
Does there exist $0\ne\eta\in L^1$ such that
$\{U^k\eta:k\in\Z\}$ is unconditional in $L^1$? Can one find, for every $r$,
functions $\eta_1,\dots,\eta_r\in L^1$ such that
\[
 \{U^k\eta_j:1\le j\le r,\ k\in\Z\}
\]
is unconditional?
\end{question}

Known correlation estimates for the horocycle map do not directly control the
signed sums in this definition.

\end{document}